\documentclass[12pt]{article}
\usepackage{geometry}
\usepackage{amssymb}
 \usepackage{xypic}
\usepackage[all,2cell]{xy}
\usepackage{graphics}
\usepackage{amsmath}
\usepackage{fullpage}
\usepackage{amsthm}
\usepackage{mathtools}
\usepackage{leftidx}
\def\cP{\mathcal P}

\newtheorem{theorem}{Theorem}[section]
\newtheorem{corollary}[theorem]{Corollary}
\newtheorem{lemma}[theorem]{Lemma}
\newtheorem{proposition}[theorem]{Proposition}

\newtheorem{definition}[theorem]{Definition}

\usepackage{hyperref}
\begin{document}

\title{A semigroup theoretic approach to the \\
Whitehead asphericity problem}
\author{Elton Pasku \\
Universiteti i Tiran\"es \\
Fakulteti i Shkencave Natyrore \\
Departamenti i Matematik\"es\\ 
Tiran\"e, Albania \\
elton.pasku@fshn.edu.al}

\date{}

\maketitle

\begin{abstract}

The Whitehead asphericity problem, regarded as a problem of combinatorial group theory, asks whether any subpresentation of an aspherical group presentation is also aspherical. This is a long standing open problem which has attracted a lot of attention. Related to it, throughout the years there have been given several useful characterizations of asphericity which are either combinatorial or topological in nature. The aim of this paper is two fold. First, it brings in methods from semigroup theory to give a new combinatorial characterization of asphericity in terms of what we define here to be the weak dominion of a submonoid of a monoid, and uses this to give a sufficient and necessary condition under which a subpresentation of an aspherical group presentation is aspherical.
\end{abstract}

\maketitle

\section{Introduction}

A 2-dimensional CW-complex $K$ is called {\it aspherical} if $\pi_2(K)=0$. The Whitehead asphericity problem, raised as a question in \cite{JHCW}, asks whether any subcomplex of an aspherical 2-complex is also aspherical. The question can be formulated in group theoretic terms since every group presentation $\mathcal{P}$ has a geometric realisation as a 2-dimensional CW-complex $K(\mathcal{P})$ and so $\mathcal{P}$ is called aspherical if $K(\mathcal{P})$ is aspherical. A useful review of this question is in \cite{Rosebrock}.

The problem is still unsolved despite of many efforts during several decades and it certainly lies deep. An indication to the deepness of the Whitehead problem is the connections it has with other major open problems in low dimension homotopy or group theory. We will mention below two of them. First, a result of Bestvina-Brady \cite{BB} says that a positive answer to Whitehead problem would imply that another old conjecture of Eilenberg and Ganea \cite{EG} is false. This conjecture states that if a discrete group $G$ has cohomological dimension 2, then it has a 2-dimensional Eilenberg-MacLane space $K(G, 1)$. 

Second, in \cite{IK} Ivanov considers the following situation. Suppose that $\mathcal{P}_{1}=(\mathbf{x} ,\mathbf{r})$ is aspherical, $z \notin \mathbf{x}$ and $w(x,z)$ is a word in the alphabet $(\mathbf{x} \cup z)^{\pm 1}$ with nonzero sum exponent of $z$. He conjectures that $\mathcal{P}=(\mathbf{x} \cup z,\mathbf{r} \cup w(x,z))$ is aspherical if and only if the group $G$ given by $\mathcal{P}$ is torsion free, and proves that if this conjecture is false and $G$ is a counterexample, then $G$ is torsion free and the integral group ring $\mathbb{Z}G$ contains zero divisors. So the existence of a counterexample as above would be a counterexample to the Kaplansky problem on zero divisors which asks whether the group ring of a torsion free group over an integral domain can have zero divisors.

There is a large corpus of results which are related to ours and is mostly contained in \cite{BP}, \cite{BH}, \cite{AGP}, \cite{CH82}, \cite{Ger87a}, \cite{Ger87b}, \cite{GR}, \cite{How79}, \cite{How81a}, \cite{How81b}, \cite{How82}, \cite{How83}, \cite{How84}, \cite{How85}, \cite{How98}, \cite{Hue82}, \cite{I IJM}, \cite{IK}, \cite{IL} and \cite{Stefan}.

In our paper we will make use of the review paper \cite{BH} of Brown and Huebschmann which contains several key results about aspherical group presentations one of which is proposition 14 that gives sufficient and necessary conditions under which a group presentation $\mathcal{P}=( \mathbf{x},\mathbf{r} )$ is aspherical. It turns out that the asphericity of $\mathcal{P}$ is encoded in the structure of the free crossed module $(H/P,F,\delta)$ that is associated to $\mathcal{P}$. To be precise we state below proposition 14.
\begin{proposition} (Proposition 14 of \cite{BH})
Let $K(\mathcal{P})$ be the geometric realisation of a group presentation $\mathcal{P}=( \mathbf{x},\mathbf{r})$ and let $G$ be the group given by $\mathcal{P}$. The following are equivalent.
\begin{description}
\item [(i)] The 2-complex $K(\mathcal{P})$ is aspherical.
\item [(ii)] The module $\pi$ of identities for $\mathcal{P}$ is zero.
\item [(iii)] The relation module $\mathcal{N(P)}$ of $\mathcal{P}$ is a free left $\mathbb{Z}G$ module on the images of the relators $r \in \mathbf{r}$.
\item [(iv)] Any identity $Y$-sequence for $\mathcal{P}$ is Peiffer equivalent to the empty sequence.
\end{description}
\end{proposition}
The last condition is of a particular interest to us. By definition, a $Y$-sequence for $\mathcal{P}$ is a finite (possibly empty) sequence of the form $((^{u_{1}}r_{1})^{\varepsilon_{1}},...,(^{u_{n}}r_{n})^{\varepsilon_{n}})$ where $r \in \mathbf{r}$, $u$ is a word from the free group $F$ over $\mathbf{x}$ and $\varepsilon =\pm 1$. A $Y$-sequence $((^{u_{1}}r_{1})^{\varepsilon_{1}},...,(^{u_{n}}r_{n})^{\varepsilon_{n}})$ is called an identity $Y$-sequence if it is either empty or if $\prod_{i=1,n}u_{i}r_{i}^{\varepsilon_{i}}u_{i}^{-1}=1$ in $F$. The definition of Peiffer equivalence is based on Peiffer operations on $Y$-sequences and reads as follows.
\begin{itemize}
\item [(i)] An \textit{elementary Peiffer exchange} replaces an adjacent pair $((^{u}r)^{\varepsilon},((^{v}s)^{\delta})$ in a $Y$-sequence by either $((^{ur^{\varepsilon}u^{-1}v}s)^{\delta},(^{u}r)^{\varepsilon})$, or by $((^{v}s)^{\delta}, ((^{vs^{-\delta}v^{-1}u}r)^{\varepsilon})$. 
\item [(ii)] A \textit{Peiffer deletion} deletes an adjacent pair $((^{u}r)^{\varepsilon},(^{u}r)^{-\varepsilon})$ in a $Y$-sequence. 
\item [(iii)] A \textit{Peiffer insertion} is the inverse of the Peiffer deletion.
\end{itemize}
The equivalence relation on the set of $Y$-sequences generated by the above operations is called \textit{Peiffer equivalence}. We recall from \cite{BH} what does it mean for an identity $Y$-sequence $((^{u_{1}}r_{1})^{\varepsilon_{1}},...,(^{u_{n}}r_{n})^{\varepsilon_{n}})$ to have the primary identity property. This means that the indices $1,2,...,n$ are grouped into pairs $(i,j)$ such that $r_{i}=r_{j}$, $\varepsilon_{i}=-\varepsilon_{j}$ and $u_{i}=u_{j}$ modulo $N$ where $N$ is the normal subgroup of $F$ generated by $\mathbf{r}$. Proposition 16 of \cite{BH} shows that every such sequence is Peiffer equivalent to the empty sequence. Given an identity $Y$-sequence $d$ which is equivalent to the empty sequence 1, we would be interested to know what kind of insertions $((^{u}r)^{\varepsilon},(^{u}r)^{-\varepsilon})$ are used along the way of transforming $d$ to 1. It is obvious that keeping track of that information is vital to tackle the Whitehead problem.

The aim of Section \ref{poma} of the present paper is to offer an alternative way in dealing with the asphericity of a group presentation $\mathcal{P}=( \mathbf{x},\mathbf{r})$ by considering a new crossed module $(\mathcal{G}(\Upsilon),F,\tilde{\theta})$ over $F$ where $\mathcal{G}(\Upsilon)$ is the group generated by the symbols $(^{u}r)^{\varepsilon}$ subject to relations $(^{u}r)^{\varepsilon} (^{v}s)^{\delta}=(^{{u{r^{\varepsilon}}u^{-1}}v}s)^{\delta}(^{u}r)^{\varepsilon}$, the action of $F$ on $\mathcal{G}(\Upsilon)$ and the map $\tilde{\theta}$ are defined in the obvious fashion. The advantage of working with $\mathcal{G}(\Upsilon)$ is that unlike to $H/P$, in $\mathcal{G}(\Upsilon)$ the images of insertions $((^{u}r)^{\varepsilon},(^{u}r)^{-\varepsilon})$ do not cancel out and this enables us to express the asphericity in terms of such insertions. This is realized by considering the kernel $\tilde{\Pi}$ of $\tilde{\theta}$ which is the analogue of the module $\pi$ of identities for $\mathcal{P}$ in the standard theory and is not trivial when $\mathcal{P}$ is aspherical. We call $\tilde{\Pi}$ the generalized module of identities for $\mathcal{P}$. This approach turns out to be useful when we discuss in Section \ref{cha} the asphericity of a subpresentation $\mathcal{P}_{1}=(\mathbf{x},\mathbf{r}_{1})$ of an aspherical presentation $\mathcal{P}=(\mathbf{x},\mathbf{r})$ where $\mathbf{r}_{1}$ differs from $\mathbf{r}$ from a single defining relation $r_{0}$. The characterization we give in this section for the asphericity of $\mathcal{P}_{1}$ reduces the search for the asphericity into the problem of deciding whether two certain groups are isomorphic or not.

To prove our results we apply techniques from the theory of semigroup actions and to this end we use concepts like the universal enveloping group $\mathcal{G}(S)$ of a given semigroup $S$, the dominion of a subsemigroup $U$ of a semigroup $S$ and the tensor product of semigroup actions. These concepts are explained, with references, in Section \ref{ma}.

\section{Monoid actions} \label{ma}

For the benefit of the reader not familiar with monoid actions we will list below some basic notions and results that are used in the paper. For further results on the subject the reader may consult the monograph \cite{Howie}. Given $S$ a monoid with identity element 1 and $X$ a nonempty set, we say that $X$ is a \textit{left S-system} if there is an action $(s,x) \mapsto sx$ from $S \times X$ into $X$ with the properties
\begin{align*}
(st)x&=s(tx) \text{ for all } s,t \in S \text{ and } x \in X,\\
1x&=x \text{ for all } x \in X.
\end{align*}
Right $S$-systems are defined analogously in the obvious way. Given $S$ and $T$  (not necessarily different) monoids, we say that $X$ is an \textit{(S,T)-bisystem} if it is a left $S$-system, a right $T$-system, and if
\begin{equation*}
(sx)t=s(xt) \text{ for all } s \in S, t \in T \text{ and } x \in X.
\end{equation*}
If $X$ and $Y$ are both left $S$-systems, then an \textit{S-morphism} or \textit{S-map} is a map $\phi: X \rightarrow Y$ such that
\begin{equation*}
\phi(sx)=s\phi(x) \text{ for all } s \in S \text{ and } x \in X.
\end{equation*}
Morphisms of right $S$-systems and of $(S,T)$-bisystems are defined in an analogue way. 
If we are given a left $T$-system $X$ and a right $S$-system $Y$, then we can give the cartesian product $X \times Y$ the structure of an $(T,S)$-bisystem by setting
\begin{equation*}
t(x,y)=(tx,y) \text{ and } (x,y)s=(x,ys).
\end{equation*}
Let now $A$ be an $(T,U)$-bisystem, $B$ an $(U,S)$-bisystem and $C$ an $(T,S)$-bisystem. As explained above, we can give to $A \times B$ the structure of an $(T,S)$-bisystem. With this in mind we say that a $(T,S)$-map $\beta: A \times B \rightarrow C$ is a \textit{bimap} if
\begin{equation*}
\beta(au,b)=\beta(a,ub) \text{ for all } a\in A, b \in B \text{ and } u \in U.
\end{equation*}
A pair $(A\otimes_{U}B,\psi)$ consisting of a $(T,S)$-bisystem $A\otimes_{U}B$ and a bimap $\psi: A \times B \rightarrow A \otimes_{U}B$ will be called a \textit{tensor product of A and B over U} if for every $(T,S)$-bisystem $C$ and every bimap $\beta: A \times B \rightarrow C$, there exists a unique $(T,S)$-map $\bar{\beta}: A\otimes_{U}B \rightarrow C$ such that the diagram
\begin{equation*}
\xymatrix{A\times B \ar[d]_{\beta} \ar[r]^{\psi} & A \otimes_{U} B \ar[ld]^{\bar{\beta}}\\C
}
\end{equation*}
commutes.  It is proved in \cite{Howie} that $A\otimes_{U}B$ exists and is unique up to isomorphism. The existence theorem reveals that $A\otimes_{U}B=(A \times B)/\tau$ where $\tau$ is the equivalence on $A \times B$ generated by the relation
\begin{equation*}
T=\{ ((au,b),(a,ub)): a \in A, b \in B, u \in U\}.
\end{equation*}
The equivalence class of a pair $(a,b)$ is usually denoted by $a \otimes_{U}b$. To us is of interest the situation when $A=S=B$ where $S$ is a monoid and $U$ is a submonoid of $S$. Here $A$ is clearly regarded as an $(S,U)$-bisystem with $U$ acting on the right on $A$ by multiplication, and $B$ as an $(U,S)$-bisystem where $U$ acts on the left on $B$ by multiplication. 

Another concept that is important to our approach is that of the dominion which is defined in \cite{Isbell} from Isbell. By definition, if $U$ is a submonoid of a monoid $S$, then the dominion $\text{Dom}_{S}(U)$ consists of all the elements $d \in S$ having the property that for every monoid $T$ and every pair of monoid homomorphisms $f,g: S \rightarrow T$ that coincide in $U$, it follows that $f(d)=g(d)$. Related to dominions there is the well known zigzag theorem of Isbell. We will present here the Stenstrom version of it (theorem 8.3.3 of \cite{Howie}) which reads. \textit{Let $U$ be a submonoid of a monoid $S$ and let $d \in S$. Then, $d \in \text{Dom}_{S}(U)$ if and only if $d \otimes_{U}1=1\otimes_{U}d$ in the tensor product $A=S \otimes_{U}S$}. We mention here that this result holds true if $S$ turns out to be a group and $U$ a subgroup, both regarded as monoids. A key result (theorem 8.3.6 of \cite{Howie}) that is used in the next section is the fact that any inverse semigroup $U$ is absolutely closed in the sense that for every  semigroup $S$ containing $U$ as a subsemigroup, $\text{Dom}_{S}(U)=U$. It is obvious that groups are absolutely closed as special cases of inverse monoids (see \cite{epidom2}).

\section{Peiffer operations and monoid actions} \label{poma}

Before we explain how monoid actions are used to deal with the Peiffer operations on $Y$-sequences, we will introduce several monoids. 

The first one is the monoid $\Upsilon$ defined by the monoid presentation $\mathcal{MP}( Y \cup Y^{-1}, P )$ where $Y^{-1}$ is the set of group inverses of the elements of $Y$ and $P$ consists of all pairs $(ab,{^{\theta (a)}}b a)$ where $a,b \in Y\cup Y^{-1}$. 

The second one is the group $\mathcal{G}(\Upsilon)$ given by the group presentation $ (Y \cup Y^{-1}, \hat{P} )$ where $\hat{P}$ is the set of all words $ab\iota(a)\iota(^{\theta(a)}b)$ where by $\iota(c)$ we denote the inverse of $c$ in the free group over $Y\cup Y^{-1}$. Before we introduce the next two monoids and the respective monoid actions, we stop to explain that $\Upsilon$ and $\mathcal{G}(\Upsilon)$ are special cases of a more general situation. If a monoid $S$ is given by the monoid presentation $\mathcal{MP}( X, R )$, then its \textit{universal enveloping group} $\mathcal{G}(S)$ (see \cite{Bergman} and \cite{Cohn}) is defined to be the group given by the group presentation $( X, \hat{R} )$ where $\hat{R}$ consists of all words $u\iota(v)$ whenever $(u,v) \in R$ where $\iota(v)$ is the inverse of $v$ in the free group over $X$. We let for future use $\sigma: FM(X) \rightarrow S$ the respective canonical homomorphism where $FM(X)$ is the free monoid on $X$. It is easy to see that there is a monoid homomorphism $\mu_{S}: S \rightarrow \mathcal{G}(S)$ which satisfies the following universal property. For every group $G$ and monoid homomorphism $f: S \rightarrow G$, there is a unique group homomorphism $\hat{f}: \mathcal{G}(S) \rightarrow G$ such that $\hat{f} \mu_{S}=f$. This universal property is an indication of an adjoint situation. Specifically, the functor $\mathcal{G}:\mathbf{Mon} \rightarrow \mathbf{Grp}$ which maps every monoid to its universal group, is a left adjoint to the forgetful functor $U: \mathbf{Grp} \rightarrow \mathbf{Mon}$. This ensures that $\mathcal{G}(S)$ is an invariant of the presentation of $S$. 

The third monoid we consider is the submonoid $\mathfrak{U}$ of $\Upsilon$, having the same unit as $\Upsilon$, and is generated from all the elements of the form $\sigma(a)\sigma(a^{-1})$ with $a \in Y \cup Y^{-1}$. This monoid, acts on the left and on the right on $\Upsilon$ by the multiplication in $\Upsilon$. 

The last monoid considered is the subgroup $\hat{\mathfrak{U}}$ of $\mathcal{G}(\Upsilon)$ generated by $\mu(\mathfrak{U})$. Similarly to above, $\hat{\mathfrak{U}}$ acts on $\mathcal{G}(\Upsilon)$ by multiplication.

Given $\alpha=(a_{1},...,a_{n})$ an $Y$-sequence over the group presentation $\mathcal{P}=( \mathbf{x},\mathbf{r})$, then performing an elementary Peiffer operation on $\alpha$ can be interpreted in a simple way in terms of the monoids $\Upsilon$ and $\mathfrak{U}$. In what follows we will denote by $\sigma(\alpha)$ the element $\sigma(a_{1})\cdot \cdot \cdot \sigma(a_{n}) \in \Upsilon$. If $\beta=(b_{1},...,b_{n})$ is obtained from $\alpha=(a_{1},...,a_{n})$ by performing an elementary Peiffer exchange, then from the definition of $\Upsilon$, $\sigma(\alpha)=\sigma(\beta)$, therefore an elementary Peiffer exchange or a finite sequence of such has no effect on the element $\sigma(a_{1})\cdot \cdot \cdot \sigma(a_{n}) \in \Upsilon$. Before we see the effect that a Peiffer insertion in $\alpha$ has on $\sigma(\alpha)$ we need the first claim of the following.
\begin{lemma} \label{central}
The elements of $\mathfrak{U}$ are central in $\Upsilon$ and those of $\hat{\mathfrak{U}}$ are central in $\mathcal{G}(\Upsilon)$.
\end{lemma}
\begin{proof}
We see that for every $a \text{ and } b \in Y \cup Y^{-1}$, $\sigma(a)\sigma(a^{-1})\sigma(b)=\sigma(b)\sigma(a)\sigma(a^{-1})$. Indeed,
\begin{align*}
\sigma(a)\sigma(a^{-1})\sigma(b)&=~^{\theta(a)\theta(a^{-1})}{\sigma(b)}(\sigma(a)\sigma(a^{-1}))\\
&=\sigma(b)\sigma(a)\sigma(a^{-1}).
\end{align*}
Since elements $\sigma(b)$ and $\sigma(a)\sigma(a^{-1})$ are generators of $\Upsilon$ and $\mathfrak{U}$ respectively, then the first claim holds true. The second claim follows easily.
\end{proof}

If we insert $(a,a^{-1})$ at some point in $\alpha=(a_{1},...,a_{n})$ to obtain $\alpha'=(a_{1},...,a,a^{-1},...,a_{n})$, then from lemma \ref{central}, 
\begin{equation*}
\sigma(\alpha')=\sigma(\alpha) \cdot (\sigma(a)\sigma(a^{-1})),
\end{equation*}
which means that inserting $(a,a^{-1})$ inside a $Y$-sequence $\alpha$ has the same effect as multiplying the corresponding $\sigma(\alpha)$ in $\Upsilon$ by the element $\sigma(a)\sigma(a^{-1})$ of $\mathfrak{U}$. For the converse, it is obvious that any word $\beta \in FM(Y \cup Y^{-1})$ representing $\sigma(\alpha)\cdot (\sigma(a)\sigma(a^{-1}))$ is Peiffer equivalent to $\alpha$. Of course the deletion has the obvious interpretation in our semigroup theoretic terms as the inverse of the above process. We retain the same names for our semigroup operations, that is insertion for multiplication by $\sigma(a)\sigma(a^{-1})$ and deletion for its inverse. Related to these operations on the elements of $\Upsilon$ we make the following definition.
\begin{definition}
We denote by $\sim_{\mathfrak{U}}$ the equivalence relation in $\Upsilon$ generated by all pairs $(\sigma(\alpha),\sigma(\alpha)\cdot \sigma(a)\sigma(a^{-1}))$ where $\alpha \in \text{FM}(Y\cup Y^{-1})$ and $a \in Y \cup Y^{-1}$. We say that two elements $\sigma(a_{1})\cdot \cdot \cdot \sigma(a_{n})$ and $\sigma(b_{1})\cdot \cdot \cdot \sigma(b_{m})$ where $m,n \geq 0$ are \textit{Peiffer equivalent in $\Upsilon$} if they fall in the same $\sim_{\mathfrak{U}}$-class.
\end{definition}
From what we said before it is obvious that two $Y$-sequences $\alpha$ and $\beta$ are Peiffer equivalent in the usual sense if and only if $\sigma(\alpha)\sim_{\mathfrak{U}} \sigma(\beta)$. For this reason we decided to make the following convention. If $\alpha=(a_{1},...,a_{n})$ is a $Y$-sequence (resp. an identity $Y$-sequence), then its image in $\Upsilon$, $\sigma(\alpha)$ will again be called a $Y$-sequence (resp. an identity $Y$-sequence). In the future instead of working directly with an $Y$-sequence $\alpha$, we will work with its image $\sigma(\alpha)$.

We note that it should be mentioned that the study of $\sim_{\mathfrak{U}}$ might be as hard as the study of Peiffer operations on $Y$-sequences, and at this point it seems we have not made any progress at all. In fact this definition will become useful later in this section and yet we have to prove a few more things before we utilize it.

The process of inserting and deleting generators of $\mathfrak{U}$ in an element of $\Upsilon$ is related to the following new concept. Given $U$ a submonoid of a monoid $S$ and $d \in S$, then we say that $d$ belongs to the \textit{weak dominion of} $U$, shortly written as $d \in \text{WDom}_{S}(U)$, if for every group $G$ and every monoid homomorphisms $f,g:S \rightarrow G$ such that $f(u)=g(u)$ for every $u \in U$, then $f(d)=g(d)$. An analogue of the Stenstr\"{o}m version of Isbell's theorem for weak dominion holds true. The proof of the if part of its analogue is similar to that of Isbell theorem apart from some minor differences that reflect the fact that we are working with $WDom$ rather than $Dom$ and that will become clear along the proof, while the converse relies on the universal property of $\mu: S \rightarrow \mathcal{G}(S)$.  
\begin{proposition} \label{wd prop}
Let $S$ be a monoid, $U$ a submonoid and let $\hat{U}$ be the subgroup of $\mathcal{G}(S)$ generated by elements $\mu(u)$ with $u \in U$. Then $d \in \text{WDom}_{S}(U)$ if and only if $\mu(d) \in \hat{U}$.
\end{proposition}
\begin{proof}
The set $\hat{A}=\mathcal{G}(S)\otimes_{\hat{U}} \mathcal{G}(S)$ has an obvious $(\mathcal{G}(S),\mathcal{G}(S))$-bisystem structure. The free abelian group $\mathbb{Z}\hat{A}$ on $\hat{A}$ inherits a $(\mathcal{G}(S),\mathcal{G}(S))$-bisystem structure if we define
\begin{equation*}
g \cdot \sum z_{i}(g_{i} \otimes_{\hat{U}}h_{i})=\sum z_{i}(gg_{i}\otimes_{\hat{U}}h_{i}) \text{ and } \left(\sum z_{i} (g_{i} \otimes_{\hat{U}}h_{i})\right)\cdot g=\sum z_{i}(g_{i} \otimes_{\hat{U}}h_{i}g).
\end{equation*}
The set $\mathcal{G}(S) \times \mathbb{Z}\hat{A}$ becomes a group by defining 
\begin{equation*}
(g,\sum z_{i} g_{i} \otimes_{\hat{U}} h_{i})\cdot (g',\sum z'_{i} g'_{i} \otimes_{\hat{U}}h'_{i})=(gg', \sum z_{i} g_{i} \otimes_{\hat{U}} h_{i}g'+\sum z'_{i} gg'_{i} \otimes_{\hat{U}}h'_{i}).
\end{equation*}
The associativity is proved easily. The unit element is $(1,0)$ and for every $(g,\sum z_{i} g_{i} \otimes_{\hat{U}} h_{i})$ its inverse is the element $(g^{-1},-\sum z_{i} g^{-1}g_{i} \otimes_{\hat{U}} h_{i} g^{-1})$. Let us now define
\begin{equation*}
\beta: S \rightarrow \mathcal{G}(S) \times \mathbb{Z}\hat{A} \text{ by } s\mapsto (\mu(s),0),
\end{equation*}
which is clearly a monoid homomorphism, and
\begin{equation*}
\gamma: S \rightarrow \mathcal{G}(S) \times \mathbb{Z}\hat{A} \text{ by } s \mapsto (\mu(s), \mu(s) \otimes_{\hat{U}}1-1\otimes_{\hat{U}} \mu(s)),
\end{equation*}
which is again seen to be a monoid homomorphism. These two coincide on $U$ since for every $u \in U$
\begin{equation*}
\gamma(u)=(\mu(u),\mu(u) \otimes_{\hat{U}}1-1\otimes_{\hat{U}}\mu(u))=(\mu(u),0)=\beta(u).
\end{equation*}
The last equality and the assumption that $d \in \text{WDom}_{S}(U)$ imply that $\beta(d)=\gamma(d)$, therefore
\begin{equation*}
(\mu(d),0)=(\mu(d),\mu(d)\otimes_{\hat{U}} 1-1 \otimes_{\hat{U}} \mu(d)),
\end{equation*}
which shows that $\mu(d)\otimes_{\hat{U}} 1=1 \otimes_{\hat{U}} \mu(d)$ in the tensor product $\mathcal{G}(S)\otimes_{\hat{U}} \mathcal{G}(S)$ and therefore theorem 8.3.3, \cite{Howie}, applied for monoids $\mathcal{G}(S)$ and $\hat{U}$, implies that $\mu(d) \in \text{Dom}_{\mathcal{G}(S)}(\hat{U})$. But $\text{Dom}_{\mathcal{G}(S)}(\hat{U})=\hat{U}$ as from theorem 8.3.6, \cite{Howie} every inverse semigroup is absolutely closed, whence $\mu(d) \in \hat{U}$.

Conversely, suppose that $\mu(d) \in \hat{U}$ and we want to show that $d \in \text{WDom}_{S}(U)$. Let $G$ be a group and $f,g: S \rightarrow G$ two monoid homomorphisms that coincide in $U$, therefore the group homomorphisms $\hat{f}, \hat{g}: \mathcal{G}(S) \rightarrow G$ of the universal property of $\mu$ coincide in $\hat{U}$ which, from our assumption, implies that $\hat{f}(\mu(d))=\hat{g}(\mu(d))$, and then $f(d)=g(d)$ proving that $d \in \text{WDom}_{S}(U)$.
\end{proof}

Given a presentation $\mathcal{P}=( \mathbf{x},\mathbf{r})$ for a group $G$, we consider the following crossed module. If $\mathcal{G}(\Upsilon)$ is the universal group associated with $\mathcal{P}$ and $F$ is the free group on $\mathbf{x}$, then we define 
\begin{equation*}
\tilde{\theta}: \mathcal{G}(\Upsilon) \rightarrow F \text{ by } \mu\sigma(^{u}{r})^{\varepsilon} \mapsto ur^{\varepsilon}u^{-1}.
\end{equation*}
An action of $F$ on $\mathcal{G}(\Upsilon)$ is given by $^{v}(\mu \sigma (^{u}r)^{\varepsilon})=\mu \sigma(^{vu}r)^{\varepsilon}$ for every $v \in F$ and every generator $\mu \sigma((^{u}r)^{\varepsilon})$ of $\mathcal{G}(\Upsilon)$. It is easy to check that the triple $(\mathcal{G}(\Upsilon),F,\tilde{\theta})$ is a crossed module over $F$. The elements of $\text{Ker}(\tilde{\theta})$ are central, therefore $\text{Ker}(\tilde{\theta})$ is an abelian subgroup of $\mathcal{G}(\Upsilon)$ on which $G$ acts on the left by the rule 
$$^{g} (\mu \sigma(a_{1},...,a_{n})\iota\mu \sigma (b_{1},...,b_{m}))=\mu \sigma (^{w}{a_{1}},...,^{w}{a_{n}})\iota \mu \sigma(^{w}{b_{1}},...,^{w}{b_{m}}),$$
where $w$ is a word in $FG(\mathbf{x})$ representing $g$. With this action $\text{Ker}(\tilde{\theta})$ becomes a left $G$-module which we call \textit{the generalized module of identities for} $\mathcal{P}$ and is denoted by $\tilde{\Pi}$. Also we note that $\hat{\mathfrak{U}}$ is a sub $G$-module of $\tilde{\Pi}$. The module of identities $\pi$ for $\mathcal{P}$ is obtained from $\tilde{\Pi}$ by factoring out $\hat{\mathfrak{U}}$. In terms of $\tilde{\Pi}$ and $\hat{\mathfrak{U}}$ we prove the following analogue of theorem 3.1 of \cite{papa-attach}.

\begin{theorem} \label{wdom}
The following assertions are equivalent.
\begin{itemize}
\item [(i)] The presentation $\mathcal{P}=( \mathbf{x},\mathbf{r})$ is aspherical.
\item [(ii)] For every identity $Y$-sequence $d$, $d \in \text{WDom}_{\Upsilon}(\mathfrak{U})$.
\item [(iii)] $\tilde{\Pi}=\hat{\mathfrak{U}}$.
\end{itemize}
\end{theorem}
\begin{proof}
$(i) \Rightarrow (ii)$ Let $d=\sigma(a_{1})\cdot \cdot \cdot \sigma(a_{n}) \in \Upsilon$ be any identity $Y$-sequence and as such it has to be Peiffer equivalent to 1. We proceed by showing that $d \in \text{WDom}_{\Upsilon}(\mathfrak{U})$. Let $G$ be any group and $f,g:\Upsilon \rightarrow G$ two monoid homomorphisms that coincide in $\mathfrak{U}$ and we want to show that $f(d)=g(d)$. The proof will be done by induction on the minimal number $h(d)$ of insertions and deletions needed to transform $d=\sigma(a_{1})\cdot \cdot \cdot \sigma(a_{n})$ to $1$. If $h(d)=1$, then $d \in \mathfrak{U}$ and $f(d)=g(d)$. Suppose that $h(d)=n>1$ and let $\tau$ be the first operation performed on $d$ in a series of operations of minimal length. After $\tau$ is performed on $d$, it is obtained an element $d'$ with $h(d')=n-1$. By induction hypothesis, $f(d')=g(d')$ and we want to prove that $f(d)=g(d)$. There are two possible cases for $\tau$. First, $\tau$ is an insertion and let $u=\sigma(a)\sigma(a^{-1}) \in \mathfrak{U}$ be the element inserted. It follows that $f(d')=f(d)f(u)$ and $g(d')=g(d)g(u)$, but $f(u)=g(u)$, therefore from cancellation law in the group $G$ we get $f(d)=g(d)$. Second, $\tau$ is a deletion and let $u=\sigma(a)\sigma(a^{-1}) \in \mathfrak{U}$ be the element deleted, that is $d=d'u$. It follows immediately from the assumptions that $f(d)=g(d)$ proving that $d \in \text{WDom}_{\Upsilon}(\mathfrak{U})$.

$(ii) \Rightarrow (iii)$ Let $\tilde{d} \in \tilde{\Pi}$. We may assume without loss of generality that no $\iota(\mu\sigma(^{u}{r})^{\varepsilon})$ is represented in $\tilde{d}$ for if there is any such occurrence, we can multiply $\tilde{d}$ by $\mu\sigma((^{u}{r})^{\varepsilon}(^{u}{r})^{-\varepsilon})$ to obtain in return $\tilde{d}'$ where $\iota(\mu\sigma(^{u}{r})^{\varepsilon})$ is now replaced by $\mu\sigma((^{u}{r})^{-\varepsilon})$. It is obvious that if $\tilde{d}' \in \hat{\mathfrak{U}}$, then $\tilde{d} \in \hat{\mathfrak{U}}$ and conversely. Let now $d$ be any preimage of $\tilde{d}$ under $\mu$. It is clear that $d$ is an identity $Y$-sequence and as such $d \in \text{WDom}_{\Upsilon}(\mathfrak{U})$. Then proposition \ref{wd prop} implies that $\tilde{d}=\mu(d) \in \hat{\mathfrak{U}}$.

$(iii) \Rightarrow (i)$ Assume that $\tilde{\Pi}=\hat{\mathfrak{U}}$ and we want to show that any identity $Y$-sequence $d$ is Peiffer equivalent to 1. From the assumption for $d$ we have that $\mu(d) \in \hat{\mathfrak{U}}$ and then proposition \ref{wd prop} implies that $d \in \text{WDom}_{\Upsilon}(\mathfrak{U})$. Consider the group $H/P$ as a quotient of $\mathcal{G}(\Upsilon)$ obtained by identifying $\iota(\mu\sigma({^{u}{r}}))$ with $\mu\sigma((^{u}{r})^{-1})$ and let $\nu:\mathcal{G}(\Upsilon) \rightarrow H/P$ be the respective quotient morphism. Writing $\tau$ for the zero morphism from $\Upsilon$ to $H/P$, we see that $\tau$ and the composition $\nu\mu$ coincide in $\mathfrak{U}$, therefore since $d \in \text{WDom}_{\Upsilon}(\mathfrak{U})$, it follows that $\nu\mu(d)=1$ in $H/P$. The asphericity of $\mathcal{P}$ now follows from theorem 2.7, p.71 of \cite{cha}.
\end{proof} 

Before we prove our next result we recall the definition of the relation module $\mathcal{N}(\mathcal{P})$. Given $\mathcal{P}=( \mathbf{x},\mathbf{r} )$ a presentation for a group $G$, we let $\alpha: FG(\mathbf{x}) \rightarrow G$ and $\beta: N \rightarrow N/[N,N]$ be the canonical homomorphisms where $N$ is the normal closure of $\mathbf{r}$ in $FG(\mathbf{x})$ and $[N,N]$ its commutator subgroup. There is a well defined $G$-action on $\mathcal{N}(\mathcal{P})=N/[N,N]$ given by
\begin{equation*}
w^{\alpha}\cdot s^{\beta}=(w^{-1}sw)^{\beta}
\end{equation*}
for every $w\in FG(\mathbf{x})$ and $s \in N$. This action extends to an action of $\mathbb{Z}G$ over $\mathcal{N}(\mathcal{P})$ by setting
\begin{equation*}
(w_{1}^{\alpha} \pm w_{2}^{\alpha})\cdot s^{\beta}=(w_{1}^{-1}sw_{1}w_{2}^{-1}s^{\pm1}w_{2})^{\beta}.
\end{equation*}
When $\cP$ is aspherical, the basis of $\mathcal{N}(\mathcal{P})$ as a free $\mathbb{Z}G$ module is the set of elements $r^{\beta}$ with $r \in \mathbf{r}$.

\begin{proposition} \label{free}
If $\mathcal{P}$ is aspherical, then $\hat{\mathfrak{U}}$ is a free $G$-module with bases equipotent to the set $\mathbf{r}$. 
\end{proposition}
\begin{proof}
The result follows if we show that $\hat{\mathfrak{U}} \cong \mathcal{N}(\mathcal{P})$ as $G$-modules. For this we define
$$\Omega: \mathcal{N}(\mathcal{P}) \rightarrow \hat{\mathfrak{U}}$$
on free generators by $r^{\beta} \mapsto \mu \sigma (rr^{-1})$ which is clearly well defined and a surjective morphism of $G$-modules. Now we prove that $\Omega$ is injective. Let 
$$\xi= \sum_{i=1}^{n} u_{i}^{\alpha}\cdot r_{i}^{\beta} - \sum_{j=n+1}^{m} v_{j}^{\alpha} \cdot r_{j}^{\beta} \in \text{Ker}(\Omega),$$
which means that
\begin{equation} \label{ker}
\prod_{i=1}^{n} \mu \sigma (^{u_{i}}r_{i}(^{u_{i}}r_{i})^{-1}) \iota \left( \prod_{j=n+1}^{m} \mu \sigma (^{v_{j}}r_{j}(^{v_{j}}r_{j})^{-1})  \right)=1.
\end{equation}
To prove that $\xi=0$ we will proceed as follows. Define 
\begin{equation*}
\gamma: FM(Y \cup Y^{-1}) \rightarrow \mathcal{N}(\mathcal{P}) 
\end{equation*}
on free generators as follows
\begin{equation*}
(^{u}r)^{\varepsilon} \mapsto u^{\alpha}\cdot r^{\beta}.
\end{equation*}
It is easy to see that $\gamma$ is compatible with the defining relations of $\Upsilon$, hence there is $g:\Upsilon \rightarrow \mathcal{N}(\mathcal{P})$ and then the universal property of $\mu$ implies the existence of $\hat{g}: \mathcal{G}(\Upsilon) \rightarrow \mathcal{N}(\mathcal{P})$ such that $\hat{g}\mu=g$. If we apply now $\hat{g}$ on both sides of (\ref{ker}) obtain
$$2 \cdot \sum_{i=1}^{n} u_{i}^{\alpha}\cdot r_{i}^{\beta} - 2 \cdot \sum_{j=n+1}^{m} v_{j}^{\alpha} \cdot r_{j}^{\beta}=0,$$
proving that $\xi=0$.
\end{proof}

Let $\chi: \mathcal{G}(\Upsilon) \rightarrow \hat{\mathfrak{U}}$ be the map defined on generators by $\mu \sigma((^{u}r)^{\varepsilon}) \mapsto \mu \sigma((^{u}r)^{\varepsilon} (^{u}r)^{-\varepsilon})$. It is easy to see that this is a well defined homomorphism of groups. Also we see that $\chi(\mu \sigma((^{u}r)^{\varepsilon}(^{u}r)^{-\varepsilon}))=\mu \sigma((^{u}r)^{\varepsilon}(^{u}r)^{-\varepsilon})^{2}$. This follows easily from the fact that $\mu \sigma ((^{u}r)^{\varepsilon} (^{u}r)^{-\varepsilon})=\mu \sigma ((^{u}r)^{-\varepsilon} (^{u}r)^{\varepsilon})$. The latter can be seen by taking the conjugate $\iota\mu\sigma((^{u}r)^{\varepsilon}) \cdot \mu \sigma ((^{u}r)^{\varepsilon} (^{u}r)^{-\varepsilon}) \cdot \mu \sigma((^{u}r)^{\varepsilon})$ and then using the fact that $\mu \sigma ((^{u}r)^{\varepsilon} (^{u}r)^{-\varepsilon})$ is central.

We denote by $\mathfrak{J}$ the image of the restriction of $\chi$ on $\tilde{\Pi}$ and let $\bar{\chi}=\chi |_{\tilde{\Pi}}$, therefore we have the short exact sequence of $G$-modules
\begin{equation} \label{ses}
\xymatrix{0 \ar[r] & K \ar[r]^{\iota} & \tilde{\Pi} \ar[r]^{\bar{\chi}} & \mathfrak{J} \ar[r] & 0},
\end{equation}
where $K$ is the kernel of $\bar{\chi}$.

Define the following subset of $\hat{\mathfrak{U}}$
$$\mathfrak{B}=\{ ^{a}\mu \sigma (rr^{-1})| a \in G \text{ and } r \in \mathbf{r}\}.$$
We have already seen in proposition \ref{free} that when $\mathcal{P}$ is aspherical, then $\hat{\mathfrak{U}}$ is a free abelian group with bases equipotent to $\mathfrak{B}$.

\begin{lemma} \label{sub}
If $\hat{\mathfrak{U}}$ is a free abelian group with bases $\mathfrak{B}$, then $\mathfrak{J}$ is isomorphic to $\hat{\mathfrak{U}}$.
\end{lemma}
\begin{proof}
Before we prove that $\mathfrak{J}$ and $\hat{\mathfrak{U}}$ are isomorphic, we will find a bases $\mathfrak{C}$ for $\mathfrak{J}$ and to this end we will use the standard scheme of finding the bases for a subgroup of a free abelian group with known bases. Assume that $G \times \mathbf{r}$ to which $\{^{a}\mu \sigma (rr^{-1})| a \in G \text{ and } r \in \mathbf{r}\}$ is bijective to is well ordered and let $\leq$ be the assumed order. Let 
$$F_{(a,r)}=\langle ^{b}\mu \sigma (ss^{-1})| (b,s) \in G \times \mathbf{r} \text{ such that } (b,s) \leq (a,r) \rangle,$$
be the subgroup of $\hat{\mathfrak{U}}$ generated by all those bases elements with label $(b,s)$ less or equal to the given $(a,r)$. Also we let
$$F_{(a,r)}'=\langle ^{b}\mu \sigma (ss^{-1})| (b,s) \in G \times \mathbf{r} \text{ such that } (b,s) < (a,r) \rangle.$$
Let now 
$$H_{(a,r)}=F_{(a,r)} \cap \mathfrak{J} \text{ and } H'_{(a,r)}=F'_{(a,r)} \cap \mathfrak{J}.$$
It is clear that an element of infinite order of $H_{(a,r)}$ is $^{a}\mu \sigma (rr^{-1})^{2}$ and that $^{a}\mu \sigma (rr^{-1})^{2} \notin H'_{(a,r)}$. To see the latter we assume for absurd that 
$$^{a}\mu \sigma (rr^{-1})^{2}= \prod_{i=1}^{n} {^{a_{i}}}\mu \sigma(r_{i}r_{i}^{-1})  \cdot \iota \left( \prod_{j=1}^{m} {^{a_{j}}}\mu \sigma(r_{j}r_{j}^{-1}) \right),$$
where $(a_{i},r_{i})  \neq (a,r)$ for all $1\leq i \leq n$ and $(a_{j},r_{j})  \neq (a,r)$ for all $1\leq j \leq m$. But this is impossible since $\hat{\mathfrak{U}}$ is a free abelian group on $\mathfrak{B}$. So it remains that $^{a}\mu \sigma (rr^{-1})^{2} \notin H'_{(a,r)}$. This implies that any generator of $H_{(a,r)}/H'_{(a,r)}$ can be written as 
$$^{a}\mu \sigma (rr^{-1})^{k_{(a,r)}}h_{(a,r)}H'_{(a,r)}$$
where $k_{(a,r)} \in \mathbb{Z}^{\ast}$ and $h_{(a,r)}$ is a product of the form 
$$\prod_{i=1}^{n} {^{a_{i}}}\mu \sigma(r_{i}r_{i}^{-1})  \cdot \iota \left( \prod_{j=1}^{m} {^{a_{j}}}\mu \sigma(r_{j}r_{j}^{-1}) \right),$$
where $(a_{i},r_{i})  \neq (a,r)$ for all $1\leq i \leq n$ and $(a_{j},r_{j})  \neq (a,r)$ for all $1\leq j \leq m$. The general scheme for constructing a bases $\mathfrak{C}$ for $\mathfrak{J}$ as an abelian group shows that $\mathfrak{C}=\{ ^{a}\mu \sigma (rr^{-1})^{k_{(a,r)}}h_{(a,r)}| a \in G, r \in \mathbf{r}\}$. Now we can define homomorphisms
$$\psi: \hat{\mathfrak{U}} \rightarrow \mathfrak{J}$$
by
$$^{a}\mu \sigma (rr^{-1}) \mapsto {^{a}}\mu \sigma (rr^{-1})^{k_{(a,r)}}h_{(a,r)},$$
and
$$\psi' :\mathfrak{J} \rightarrow \hat{\mathfrak{U}}$$
by
$$^{a}\mu \sigma (rr^{-1})^{k_{(a,r)}}h_{(a,r)} \mapsto {^{a}}\mu \sigma (rr^{-1}).$$
It is obvious that $\psi$ and $\psi'$ are inverses of each other proving that $\hat{\mathfrak{U}}\cong \mathfrak{J}$.
\end{proof}

\begin{theorem} \label{U}
If $\hat{\mathfrak{U}}$ is a free abelian group with bases $\mathfrak{B}$, then the presentation $\mathcal{P}$ is aspherical, and conversely. 
\end{theorem}
\begin{proof}
The assumption on the freeness of $\hat{\mathfrak{U}}$ on $\mathfrak{B}$ implies by lemma \ref{sub} that $\mathfrak{J}$ is free on $\mathfrak{C}$ and then we have a section $s$ of $\bar{\chi}$. It follows that $\tilde{\Pi}=K \oplus s(\mathfrak{J})$ and that $s(\mathfrak{J}) \cong \mathfrak{J}$. Applying the result of lemma \ref{sub} again we have the isomorphism $\tilde{\Pi} \cong K \oplus \hat{\mathfrak{U}}$. If now 
$$d=\prod_{i=1}^{n} \mu \sigma ((^{u_{i}}r_{i})^{\varepsilon_{i}}) \iota \left( \prod_{j=n+1}^{m} \mu \sigma ((^{v_{j}}r_{j})^{\varepsilon_{j}})  \right) \in K,$$ 
then in $\hat{\mathfrak{U}}$ we have 
\begin{equation} \label{b}
\prod_{i=1}^{n} {^{ u_{i}^{\alpha}}} \mu \sigma (r_{i}r_{i}^{-1}) \iota \left( \prod_{j=n+1}^{m} {^{v_{j}^{\alpha}}} \mu \sigma (r_{j} r_{j}^{-1})  \right) =1.
\end{equation}
The freeness of $\hat{\mathfrak{U}}$ on $\mathfrak{B}$ implies that $m=2n$ and the indices of the terms in (\ref{b}) are paired up into pairs $(i,j)$ in such a way that $1\leq i \leq n$, $n+1 \leq j \leq m$, $r_{i}=r_{j}$ and $u_{i}=v_{j}$ modulo $N$. If we multiply both sides of $d$ by
$$ \hat{u}=\prod_{j=n+1}^{m} \mu \sigma ((^{v_{j}}r_{j})^{\varepsilon_{j}}(^{v_{j}}r_{j})^{-\varepsilon_{j}}), $$
$d$ transforms into
$$d'=\prod_{i=1}^{n} \mu \sigma ((^{u_{i}}r_{i})^{\varepsilon_{i}})  \left( \prod_{j=n+1}^{m} \mu \sigma ((^{v_{j}}r_{j})^{-\varepsilon_{j}})  \right).$$
From the above we see that the identity $Y$-sequence
$$\delta'=((^{u_{1}}r_{1})^{\varepsilon_{1}},...,(^{u_{n}}r_{n})^{\varepsilon_{n}},(^{v_{n+1}}r_{n+1})^{\varepsilon_{n+1}},...,(^{v_{m}}r_{m})^{\varepsilon_{m}})$$
has the primary identity property and then $\delta'$ is Peiffer equivalent to the empty sequence. Theorem \ref{wdom} implies that $d'=\mu \sigma (\delta') \in \hat{\mathfrak{U}}$ and so $d=d'\cdot \iota(\hat{u}) \in \hat{\mathfrak{U}}$. But $\bar{\chi}(d)=0$ with $d \in \hat{\mathfrak{U}}$ is only possible when $d=0$, so $K=\{0\}$ and then $\tilde{\Pi} \cong \hat{\mathfrak{U}}$ which in turn proves that the module of identities $\pi$ for $\mathcal{P}$ is zero as $\pi \cong \tilde{\Pi}/\hat{\mathfrak{U}}$, hence we have the asphericity of $\mathcal{P}$.

The converse follows from proposition \ref{free}.
\end{proof}

\section{A characterization for the asphericity of subpresentations} \label{cha}

Let $\mathcal{P}=(\mathbf{x},\mathbf{r})$ be an aspherical group presentation and $\mathcal{P}_{1}=(\mathbf{x},\mathbf{r}_{1})$ a subpresentation of the first where $\mathbf{r}_{1}=\mathbf{r}\setminus \{r_{0}\}$ and $r_{0}\in \mathbf{r}$ is a fixed relation. We denote by $\Upsilon_{1}$, $\mathfrak{U}_{1}$ monoids associated with $\mathcal{P}_{1}$ and by $\mathcal{G}(\Upsilon_{1})$ and $\hat{\mathfrak{U}}_{1}$ their respective groups and let $\tilde{\theta}_{1}$ be the morphism of the crossed module $\mathcal{G}(\Upsilon_{1})$ whose kernel is denoted by $\tilde{\Pi}_{1}$. Also we consider $\hat{\mathfrak{A}}_{1}$ the subgroup of $\hat{\mathfrak{U}}$ generated by all $\mu\sigma(bb^{-1})$ where $b \in Y_{1} \cup Y_{1}^{-1}$. Finally note that the monomorphism $f: \Upsilon_{1} \rightarrow \Upsilon$ induced by the map $\sigma_{1}(a) \rightarrow \sigma(a)$ induces a homomorphism $\hat{f}: \mathcal{G}(\Upsilon_{1} )\rightarrow \mathcal{G}(\Upsilon)$. These data fit into a commutative diagram as depicted below.
\begin{equation*}
\xymatrix{ \mathcal{G}(\Upsilon_{1}) \ar[rr]^{\hat{f}} \ar[rd]_{\tilde{\theta}_{1}} &&
\mathcal{G}(\Upsilon) \ar[ld]^-{\tilde{\theta}} \\ & F}
\end{equation*}

An immediate corollary of proposition \ref{free} is the following.
\begin{corollary} \label{free1}
If $\mathcal{P}$ is aspherical, then the subgroup $\hat{\mathfrak{A}}_{1}$ is a free abelian with basis equipotent to the set $G \times \mathbf{r}_{1}$. 
\end{corollary}
\begin{proof}
The restriction of $\Omega$ of proposition \ref{free} on the subgroup $\mathcal{N}_{1}(\mathcal{P})$ of $\mathcal{N}(\mathcal{P})$ generated by all the elements of the form $u^{\alpha}\cdot r^{\beta}$ where $r \in \mathbf{r}_{1}$ is an isomorphism onto $\hat{\mathfrak{A}}_{1}$.
\end{proof}

Further we note that the normal closure $N_{0}$ of $r_{0}$ in $FG(\mathbf{x})$ acts on the left of $\hat{\mathfrak{U}}_{1}$ in the obvious way therefore we have the displacement subgroup $[\hat{\mathfrak{U}}_{1},N_{0}]$ of $\hat{\mathfrak{U}}_{1}$ (see \cite{NAT}) which is a normal subgroup of $\tilde{\Pi}_{1}$.

\begin{proposition} \label{quasi}
If $\mathcal{P}$ is aspherical, then there is an epimorphism from $\tilde{\Pi}_{1}/ [\hat{\mathfrak{U}}_{1},N_{0}]$ to $\hat{\mathfrak{A}}_{1}$.
\end{proposition}
\begin{proof}
First we show that $\hat{f}$ maps $\tilde{\Pi}_{1}$ onto $\hat{\mathfrak{A}}_{1}$. Let $\tilde{d}=\mu_{1}\sigma_{1}(a_{1}\cdot \cdot \cdot a_{n}) \in \tilde{\Pi}_{1}$ where as before no $a_{i}$ is equal to any $\iota(\mu_{1}\sigma_{1}(^{u}{r})^{\varepsilon})$ and assume that
\begin{align*}
\hat{f}(\tilde{d})=&(\mu\sigma(b_{1}b_{1}^{-1})\cdot \cdot \cdot \mu\sigma(b_{s}b_{s}^{-1}))\cdot (\iota(\mu\sigma(b_{s+1}b_{s+1}^{-1}))\cdot \cdot \cdot \iota(\mu\sigma(b_{r}b_{r}^{-1})))\\
&(\mu\sigma(c_{1}c_{1}^{-1})\cdot \cdot \cdot \mu\sigma(c_{t}c_{t}^{-1}))\cdot (\iota(\mu\sigma(d_{1}d_{1}^{-1}))\cdot \cdot \cdot \iota(\mu\sigma(d_{k}d_{k}^{-1}))),
\end{align*} 
where the first half involves elements from $Y_{1}\cup Y_{1}^{-1}$ and the second one is 
$$\mu\sigma(C)\iota(\mu\sigma(D))$$ 
with
$$C=c_{1}c_{1}^{-1} \cdot \cdot \cdot c_{t}c_{t}^{-1} \text{ and } D=d_{1}d_{1}^{-1} \cdot \cdot \cdot d_{k}d_{k}^{-1},$$
where $C$ and $D$ involve only elements of the form $(^{u}{r_{0}})^{\varepsilon}$ with $\varepsilon=\pm1$. Recalling from above that in $\mathcal{G}(\Upsilon)$ we have 
\begin{align*}
\mu \sigma((a_{1} \cdot \cdot \cdot a_{n}) \cdot ((b_{s+1}b_{s+1}^{-1}) \cdot \cdot \cdot (b_{r}b_{r}^{-1})) \cdot ((d_{1}d_{1}^{-1}) \cdot \cdot \cdot (d_{k}d_{k}^{-1})))\\
= \mu \sigma(((b_{1}b_{1}^{-1}) \cdot \cdot \cdot (b_{s}b_{s}^{-1})) \cdot ((c_{1}c_{1}^{-1}) \cdot \cdot \cdot (c_{t}c_{t}^{-1}))),
\end{align*}
we can apply $\hat{g}$ defined in proposition \ref{free} on both sides and get
\begin{align*}
g\sigma((a_{1} \cdot \cdot \cdot a_{n}) \cdot ((b_{s+1}b_{s+1}^{-1}) \cdot \cdot \cdot (b_{r}b_{r}^{-1})) \cdot ((d_{1}d_{1}^{-1}) \cdot \cdot \cdot (d_{k}d_{k}^{-1})))\\
=g\sigma(((b_{1}b_{1}^{-1}) \cdot \cdot \cdot (b_{s}b_{s}^{-1})) \cdot ((c_{1}c_{1}^{-1}) \cdot \cdot \cdot (c_{t}c_{t}^{-1}))).
\end{align*}
If we now write each $c_{i}=(^{u_{i}}r_{0})^{\varepsilon_{i}}$ and each $d_{j}=(^{v_{j}}r_{0})^{\delta_{j}}$ where $\varepsilon_{i}$ and $\delta_{j}=\pm1$, while we write each $a_{\ell}=(^{w_{\ell}}r_{\ell})^{\gamma_{\ell}}$ and each $b_{p}=(^{\eta_{p}}\rho_{p})^{\epsilon_{p}}$ where all $r_{\ell}$ and $\rho_{p}$ belong to $\mathbf{r}_{1}$ and $\gamma_{\ell}, \epsilon_{p}=\pm 1$, then the definition of $g$ yields
\begin{align*}
(w_{1}^{\alpha}\cdot r_{1}^{\beta}+\cdot \cdot \cdot + w_{n}^{\alpha}\cdot r_{n}^{\beta})+(2\eta_{s+1}^{\alpha} \cdot \rho_{s+1}^{\beta} + \cdot \cdot \cdot +2\eta_{r}^{\alpha} \cdot \rho_{r}^{\beta})+(2v_{1}^{\alpha}+\cdot \cdot \cdot + 2v_{k}^{\alpha})\cdot r_{0}^{\beta}\\
=(2\eta_{1}^{\alpha} \cdot \rho_{1}^{\beta} + \cdot \cdot \cdot +2\eta_{s}^{\alpha} \cdot \rho_{s}^{\beta})+(2u_{1}^{\alpha}+\cdot \cdot \cdot + 2u_{t}^{\alpha})\cdot r_{0}^{\beta}
\end{align*}
The freeness of $\mathcal{N}(\mathcal{P})$ on the set of elements $r^{\beta}$ implies in particular that
\begin{equation*}
(2v_{1}^{\alpha}+\cdot \cdot \cdot + 2v_{k}^{\alpha})\cdot r_{0}^{\beta}=(2u_{1}^{\alpha}+\cdot \cdot \cdot + 2u_{t}^{\alpha})\cdot r_{0}^{\beta}
\end{equation*}
from which we see that $k=t$, and after a rearrangement of terms $u^{\alpha}_{i}=v^{\alpha}_{i}$ for $i=1,...,k$. The already known fact that in $\mathcal{G}(\Upsilon)$, $\mu\sigma(aa^{-1})=\mu\sigma(a^{-1}a)$ and the fact that if $u^{\alpha}=v^{\alpha}$, then for every $s \in \mathbf{r}$, $\mu\sigma((^{v}s)^{\delta}(^{v}s)^{-\delta})=\mu\sigma((^{u}s)^{\delta}(^{u}s)^{-\delta})$, imply easily that
\begin{equation*}
\mu\sigma((^{v}r_{0})^{\delta}(^{v}r_{0})^{-\delta})=\mu\sigma((^{u}r_{0})^{\varepsilon}(^{u}r_{0})^{-\varepsilon}).
\end{equation*}
If we apply the latter to pairs $(c_{i}, d_{i})$ for which $u^{\alpha}_{i}=v^{\alpha}_{i}$, we get that $\mu\sigma(C)\iota(\mu\sigma(D))=1$ which shows that $\hat{f}(\tilde{d}) \in \hat{\mathfrak{A}}_{1}$. 

Second, $\hat{f}$ induces $\hat{\varphi}: \tilde{\Pi}_{1}/ [\hat{\mathfrak{U}}_{1},N_{0}] \rightarrow \hat{\mathfrak{A}}_{1}$ because if $\iota(\hat{u}_{1}) ^{n_{0}}\hat{u}_{1}$ is any generator of $[\hat{\mathfrak{U}}_{1},N_{0}]$ and if $\hat{f}(\hat{u}_{1})=\hat{u}\iota(\hat{v})$ where $\hat{u}, \hat{v} \in \hat{\mathfrak{A}}_{1}$, then 
$$\hat{f}(\iota(\hat{u}_{1}) ^{n_{0}}\hat{u}_{1})=\hat{v}\iota(\hat{u}) \cdot ^{n_{0}}{\hat{u}} ^{n_{0}}{\iota(\hat{v})}=1,$$
since for every $\hat{u} \in \hat{\mathfrak{U}}$, $^{n_{0}}\hat{u}=\hat{u}$ in $\mathcal{G}(\Upsilon)$. 
\end{proof}

The following reduces the search for the asphericity of the subpresentation $\mathcal{P}_{1}$ in proving that the groups $\tilde{\Pi}_{1}/ [\hat{\mathfrak{U}}_{1},N_{0}]$ and $\hat{\mathfrak{A}}_{1}$ are isomorphic. 

\begin{theorem} \label{subp}
The subpresentation $\mathcal{P}_{1}=(\mathbf{x},\mathbf{r}_{1})$ is aspherical if and only if the groups $\tilde{\Pi}_{1}/ [\hat{\mathfrak{U}}_{1},N_{0}]$ and $\hat{\mathfrak{A}}_{1}$ are isomorphic under $\hat{\varphi}$. 
\end{theorem}
\begin{proof}
Under the assumption that $\mathcal{P}_{1}$ is aspherical we have to show that $\hat{\varphi}$ has an inverse. Indeed, since from corollary \ref{free1}, $\hat{\mathfrak{A}}_{1}$ is free abelian with basis the set of elements of the form $\mu \sigma (^{u}r (^{u}r)^{-1})$ we can define 
$$\hat{\varphi}': \hat{\mathfrak{A}}_{1} \rightarrow \tilde{\Pi}_{1}/ [\hat{\mathfrak{U}}_{1},N_{0}] \text{ by } \mu \sigma (^{u}r (^{u}r)^{-1}) \mapsto \mu_{1} \sigma_{1} (^{u}r (^{u}r)^{-1})[\hat{\mathfrak{U}}_{1},N_{0}],$$
which is well defined and a right inverse of $\hat{\varphi}$ since from the definition, $\hat{\varphi}\hat{\varphi}'=id_{\hat{\mathfrak{A}}_{1}}$. To prove that $\hat{\varphi}'$ is a left inverse of $\hat{\varphi}$ we recall from theorem \ref{wdom} that $\tilde{\Pi}_{1} = \hat{\mathfrak{U}}_{1}$ hence for every generator $\mu_{1}\sigma_{1}(^{u}r(^{r}r)^{-1}) \cdot [\hat{\mathfrak{U}}_{1},N_{0}]$ of $\tilde{\Pi}_{1}/ [\hat{\mathfrak{U}}_{1},N_{0}]$ we have, 
\begin{align*}
\hat{\varphi}'\hat{\varphi}(\mu_{1}\sigma_{1}(^{u}r(^{r}r)^{-1})\cdot [\hat{\mathfrak{U}}_{1},N_{0}])&=\hat{\varphi}'\hat{f}(\mu_{1}\sigma_{1}(^{u}r(^{r}r)^{-1}))\\
&=\hat{\varphi}'\mu \sigma(^{u}r(^{r}r)^{-1})\\
&=\mu_{1} \sigma_{1} (^{u}r (^{u}r)^{-1})[\hat{\mathfrak{U}}_{1},N_{0}],
\end{align*}
proving that $\hat{\varphi}'\hat{\varphi}=id_{\tilde{\Pi}_{1}/ [\hat{\mathfrak{U}}_{1},N_{0}]}$.

Conversely, let $d$ be any identity $Y_{1}$-sequence and $u,v \in \mathfrak{U}_{1}$ such that 
$$\hat{\varphi}(\mu_{1}(u) \iota \mu_{1}(v))\cdot [\hat{\mathfrak{U}}_{1},N_{0}])=\hat{\varphi}(\mu_{1}\sigma_{1}(d) \cdot  [\hat{\mathfrak{U}}_{1},N_{0}]).$$
Such elements exist because $\hat{f}(\mu_{1}\sigma_{1}(d)) \in \hat{\mathfrak{A}}_{1}.$ But $\hat{\varphi}$ is an isomorphism, hence there is $K \in [\hat{\mathfrak{U}}_{1},N_{0}]$ such that 
\begin{equation} \label{K}
\mu_{1}\sigma_{1}(d) = \mu_{1}(u) \iota \mu_{1}(v) \cdot K,
\end{equation}
where $K$ equals to a product of generators $^{n_{0,i}}\hat{u}_{i}\iota(\hat{u}_{i})$ since conjugating a generator by an element of $\tilde{\Pi}_{1}$ does not alter the generator. If we apply on both sides of (\ref{K}) the canonical map $\nu_{1}: \mathcal{G}(\Upsilon_{1}) \rightarrow H_{1}/P_{1}$ where $H_{1}/P_{1}$ is the usual free crossed module for $\mathcal{P}_{1}$, then we see that $dP_{1}$ is trivial, thereby proving the claim.
\end{proof}

\end{document}